\documentclass[10pt,reqno,a4paper]{amsart}
\usepackage{amsaddr}
\usepackage{graphicx}
\usepackage[nocompress]{cite}
\usepackage{tikz-cd}

\usepackage{mathrsfs}
\usepackage{amsopn,amssymb,amsmath}
\usepackage{url}
\usepackage{subcaption}
\usepackage{verbatim}
\usepackage[hidelinks]{hyperref}

\theoremstyle{definition}
\newtheorem{theorem}{Theorem}[section]
\newtheorem{prop}[theorem]{Proposition}
\newtheorem{defn}[theorem]{Definition}
\newtheorem{cor}[theorem]{Corollary}
\newtheorem{eg}[theorem]{Example}
\newtheorem{lemma}[theorem]{Lemma}
\newtheorem{remark}[theorem]{Remark}
\numberwithin{equation}{section}

\DeclareMathOperator{\rank}{rank}

\DeclareMathOperator{\id}{id}

\DeclareMathOperator{\Ab}{Ab}
\newcommand{\Z}{\mathbb{Z}}

\begin{document}

\title{Weighted Fundamental Group}

\author{Chengyuan Wu\textsuperscript{*}}
\address{Department of Mathematics, National University of Singapore, Singapore 119076, Singapore}
\email{wuchengyuan@u.nus.edu}
\thanks{\text{*}First authors. The project was supported in part by the Singapore Ministry of Education research grant (AcRF Tier 1 WBS No.~R-146-000-222-112). The first author was supported in part by the President's Graduate Fellowship of National University of Singapore. The second author was supported by the Postdoctoral International Exchange Program of China 2019 project from The Office of China Postdoctoral Council, China Postdoctoral Science Foundation. The third author was supported by Natural Science Foundation of China (NSFC grant no.\ 11971144) and High-level Scientific Research Foundation of Hebei Province. The fourth author was supported by Nanyang Technological University Startup Grants M4081842, Singapore Ministry of Education Academic Research Fund Tier 1 RG31/18, Tier 2 MOE2018-T2-1-033.}

\author{Shiquan Ren\textsuperscript{*}}
\address{Yau Mathematical Sciences Center, Tsinghua University, Beijing 100084, China}
\email{srenmath@tsinghua.edu.cn}

\author{Jie Wu\textsuperscript{*}}
\address{School of Mathematical Sciences, Hebei Normal University, Hebei 050024, China}
\email{wujie@hebtu.edu.cn}

\author{Kelin Xia\textsuperscript{*}}
\address{(a) Division of Mathematical Sciences, School of Physical and Mathematical Sciences, Nanyang Technological
University, Singapore 637371, Singapore \\(b) School of Biological Sciences, Nanyang Technological University, Singapore 637371, Singapore}
\email{xiakelin@ntu.edu.sg}






\begin{abstract}
In this paper, we develop and study the theory of weighted fundamental groups of weighted simplicial complexes. When all weights are 1, the weighted fundamental group reduces to the usual fundamental group as a special case. We also study weighted versions of classical theorems like van Kampen's theorem. In addition, we also investigate the abelianization, lower central series and applications of weighted fundamental groups.
\end{abstract}

\maketitle
\section{Introduction}
Weighted structures, such as weighted graphs, are common in mathematics. The addition of weights to a mathematical object often adds new information to the object. Other than weighted graphs, weights have also been considered on hypergraphs \cite{ren2018cohomology,ihler1992modeling,lee2002algorithms} and simplicial complexes \cite{Dawson1990,ren2018weighted,ren2017further,wu2018weighted}.

The fundamental group is an important topological invariant. In this paper, our goal is to study the weighted fundamental group of a weighted simplicial complex. Intuitively, the weighted fundamental group should contain information about the weights of the simplicial complex. In addition, the weighted fundamental group should reduce to the usual fundamental group as a special case. Ideally, the weighted fundamental group should also satisfy (weighted versions of) classical theorems like van Kampen's theorem. We show that our definition of the weighted fundamental group fulfills the above requirements. Our approach is to modify the description of fundamental groups using maximal trees, by introducing weights on edges (1-simplices).

To our knowledge, there is no existing literature on weighted fundamental groups of weighted simplicial complexes. In \cite{bogley2000weighted}, a weighted combinatorial group theory is defined, however this is in another context of omega-groups \cite{sieradskiomega} and wild spaces, which the authors define to be metric spaces with arbitrarily small essential features. Weighted (co)homology of simplicial complexes has been previously studied in \cite{Dawson1990,ren2018weighted,ren2017further,wu2018weighted}.

In contrast with the classical case, our weighted fundamental group depends on the choice of maximal tree in general. This may have potential applications in situations where each maximal tree needs to be distinguished. In the special case when all weights are 1, the weighted fundamental group is independent of choice of maximal tree (Corollary \ref{cor:indepvert}). In addition, in the case of weighted graphs, the weighted fundamental group is independent of choice of maximal tree when all weights are equal (Corollary \ref{cor:indeptreegraph}). In Section \ref{sec:weightedgraph}, we study weighted graphs, which are 1-dimensional weighted simplicial complexes. We state and prove the weighted van Kampen Theorem in Section \ref{sec:vankampen}. In Section \ref{sec:lowercs}, we study the lower central series for certain cases of weighted fundamental groups. Finally, in Sections \ref{sec:track} and \ref{sec:chem} we outline some possible applications of weighted fundamental groups.

\section{Weighted Fundamental Group}
In this paper, we let $K$ be a path-connected (abstract) simplicial complex, and let $A$ be a fixed maximal tree in $K$. We let the set of vertices of $K$ be totally ordered. We write $v_0v_1\dots v_n$ to denote a simplex $[v_0,\dots,v_n]$ spanned by vertices $v_0,\dots,v_n$ in $K$. We remark that our definition is independent of vertex ordering, the vertex ordering is just to simplify the eventual presentation of the weighted fundamental group.

\begin{defn}[Weighted simplicial complex]
Let $w$ be a function from the 1-simplices of $K$ to $\Z$. We call $w$ a \emph{weight function}. For a 1-simplex $\sigma\in K$, we say that $w(\sigma)$ is the \emph{weight} of $\sigma$. We call the triple $(K,w,A)$ a \emph{weighted simplicial complex}, or WSC for short.
\end{defn}

We define a group $\pi_1(K,w,A)$ combinatorially as follows. We call $\pi_1(K,w,A)$ the \emph{weighted fundamental group} of $K$ induced by the weight function $w$ and the maximal tree $A$.

\begin{defn}
\label{defn:originalrelations}
The generators of $\pi_1(K,w,A)$ are given by the letters
\[g_{ab},\]
where $ab$ is a 1-simplex of $K$, and defining relations given by:
\begin{enumerate}
\item $g_{ab}^{w(ab)}=1$ if the 1-simplex $ab$ lies in $A$,
\item $g_{ab}^{w(ab)}=g_{av}^{w(av)}g_{vb}^{w(vb)}$ if $avb$ is a 2-simplex of $K$,
\item $g_{ba}=g_{ab}^{-1}$.
\end{enumerate}
\end{defn}

\begin{prop}
\label{prop:vertindep}
Definition \ref{defn:originalrelations} is well-defined and independent of the vertex ordering on $K$.
\end{prop}
\begin{proof}
We note that $g_{ab}^{w(ab)}=1$ is equivalent to $g_{ba}^{w(ba)}=1$ due to relation (3). We can also check that for all possible orderings of $avb$ (i.e. $abv$, $bav$, etc.) lead to the same relation (2), again due to relation (3). Finally, $g_{ab}=g_{ba}^{-1}$ is equivalent to relation (3). We observe that any permutation of the vertex ordering leads to the same generators (after relabelling) satisfying the same relations.
\end{proof}

We can now remove some unnecessary generators. Since $g_{ba}=g_{ab}^{-1}$, we only need to introduce a generator $g_{ab}$ for each 1-simplex $ab$ of $K$ with $a<b$. Due to Proposition \ref{prop:vertindep}, we can restrict relation (2) to 2-simplices $avb$ of $K$ with $a<v<b$. This leads us to the following equivalent definition, which will also be independent of vertex ordering on $K$.

\begin{defn}
\label{defn:relations}
The generators of $\pi_1(K,w,A)$ are given by the letters
\[g_{ab},\]
where $ab$ is a 1-simplex of $K$ with $a<b$, and defining relations given by:
\begin{enumerate}
\item $g_{ab}^{w(ab)}=1$ if the 1-simplex $ab$ lies in $A$,
\item $g_{ab}^{w(ab)}=g_{av}^{w(av)}g_{vb}^{w(vb)}$ if $a<v<b$ and $avb$ is a 2-simplex of $K$.
\end{enumerate}
\end{defn}

\begin{remark}
We remark that in general, Definition \ref{defn:relations} depends on the choice of maximal tree $A$. However, in the special case where all weights are 1, the definition is independent of choice of maximal tree. This is a corollary of the following theorem.
\end{remark}

\begin{theorem}
\label{thm:allweights1}
When $w(\sigma)\equiv 1$ for all 1-simplices $\sigma\in K$, $\pi_1(K,w,A)$ is isomorphic to the usual fundamental group $\pi_1(K,v_0)$, where $v_0$ is a vertex of the path-connected simplicial complex $K$.
\end{theorem}
\begin{proof}
When $w(\sigma)\equiv 1$ for all 1-simplices $\sigma\in K$, Definition \ref{defn:relations} reduces to a well-known combinatorial description of the usual fundamental group (cf.\ \cite[pp.~133--135]{armstrong2013basic}, \cite[p.~240]{hilton1960introduction}). 
\end{proof}

\begin{cor}
\label{cor:indepvert}
When $w(\sigma)\equiv 1$ for all 1-simplices $\sigma\in K$, Definition \ref{defn:relations} is independent of choice of maximal tree.
\end{cor}
\begin{proof}
This is a consequence of Theorem \ref{thm:allweights1}, since it is known that the usual fundamental group is independent of choice of maximal tree $A$ (cf.\ \cite[p.~135]{armstrong2013basic}).
\end{proof}

\begin{remark}
It is known that if $K$ is path-connected, the group $\pi_1(K,v_0)$ is, up to isomorphism, independent of the choice of basepoint $v_0$ (cf.\ \cite[p.~28]{Hatcher2002}). Hence, in this case we may use the abbreviated notation $\pi_1(K)$ to denote $\pi_1(K,v_0$).
\end{remark}

We show some examples where $\pi_1(K,w,A)$ is different from the usual fundamental group $\pi_1(K)$. We write $w_{ij}$ for $w(v_iv_j)$ and $g_{ij}$ for $g_{v_iv_j}$. We write $\Z/n$ for the cyclic group $\Z/n\Z$. We use the notation $\prod^*_{i\in I}G_i$ to denote the free product of a family of groups $\{G_i\}_{i\in I}$.

\begin{eg}
\label{eg:circleK}
Let $K$ be the simplicial complex (homotopy equivalent to the circle $S^1$) shown in Figure \ref{fig:circle}, where the maximal tree $A$ is marked in bold.
\begin{figure}[htbp]
\begin{center}
\begin{tikzpicture}
\filldraw 
(0,0) circle (2pt) node[align=left,below] {$v_0$}
(2,0) circle (2pt) node[align=left,below] {$v_1$}  
(1,1.7) circle (2pt) node[align=left,above] {$v_2$};
\draw[line width=2pt] (0,0)--(2,0);
\draw[line width=2pt] (2,0)--(1,1.7);
\draw (1,1.7)--(0,0);
\end{tikzpicture}
\caption{The simplicial complex $K$ with maximal tree $A$ (marked in bold).}
\label{fig:circle}
\end{center}
\end{figure}
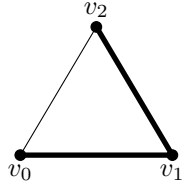

By definition, 
\begin{equation*}
\begin{split}
\pi_1(K,w,A)&=\langle g_{01},g_{02},g_{12}\mid g_{01}^{w_{01}}=1, g_{12}^{w_{12}}=1\rangle\\
&\cong\Z*(\Z/w_{01})*(\Z/w_{12}).
\end{split}
\end{equation*}

In particular, if $w_{01}\neq\pm 1$ or $w_{12}\neq\pm 1$, then $\pi_1(K,w,A)\not\cong\pi_1(K)=\Z$.
\end{eg}

\begin{remark}
\label{remark:choicemaxtree}
By observing Example \ref{eg:circleK}, we can see that the weighted fundamental group depends on the choice of maximal tree. For instance, if $B=\{[v_0,v_1],[v_0,v_2],[v_0],[v_1],[v_2]\}$ is chosen as the maximal tree of $K$, then
\begin{equation*}
\pi_1(K,w,B)\cong\Z*(\Z/w_{01})*(\Z/w_{02})
\end{equation*}
instead.
\end{remark}

\begin{eg}
\label{eg:circleL}
Let $L$ be the 2-simplex shown in Figure \ref{fig:2simplex}, with maximal tree $A$ marked in bold.
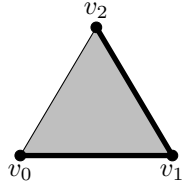
\begin{figure}[htbp]
\begin{center}
\begin{tikzpicture}
\fill[fill=lightgray]
(0,0)  
-- (2,0)
-- (1,1.7);
\filldraw 
(0,0) circle (2pt) node[align=left,below] {$v_0$}
(2,0) circle (2pt) node[align=left,below] {$v_1$}  
(1,1.7) circle (2pt) node[align=left,above] {$v_2$};
\draw[line width=2pt] (0,0)--(2,0);
\draw[line width=2pt] (2,0)--(1,1.7);
\draw (1,1.7)--(0,0);
\end{tikzpicture}
\caption{The simplicial complex $L$ with maximal tree $A$ (marked in bold).}
\label{fig:2simplex}
\end{center}
\end{figure}
\end{eg}

We have
\begin{equation}
\begin{split}
\pi_1(L,w,A)&=\langle g_{01},g_{02},g_{12}\mid g_{01}^{w_{01}}=1,g_{12}^{w_{12}}=1,g_{02}^{w_{02}}=g_{01}^{w_{01}}g_{12}^{w_{12}}=1\rangle\\
&\cong\Z/w_{01}*\Z/w_{02}*\Z/w_{12}.
\end{split}
\end{equation}

In particular, if any of $w_{01},w_{02},w_{12}$ is not equal to $\pm 1$, then $\pi_1(L,w,A)\not\cong\pi_1(L)=1$.

\begin{theorem}
\label{thm:exactly2formula}
Let $(K,w,A)$ be a weighted simplicial complex such that for each 2-simplex $avb\in K$, exactly 2 of the 1-simplices $ab$, $av$, $vb$ lie in the maximal tree $A$.

Then,
\begin{equation*}
\pi_1(K,w,A)\cong\prod^*_{ab\in A}\Z/w(ab)*\prod^*_{\substack{ab\in K\setminus A \\ \text{and $ab$ is a face of}\\ \text{some 2-simplex of $K$}}}\Z/w(ab)*\prod^*_{\substack{ab\in K\setminus A \\ \text{and $ab$ is not a face of}\\ \text{any 2-simplex of $K$}}}\Z.
\end{equation*}
\end{theorem}
\begin{proof}
Each generator $g_{ab}$ of $\pi_1(K,w,A)$ falls into 1 of the following 3 cases:
\begin{enumerate}
\item If $ab\in A$, then $g_{ab}$ satisfies the relation $g_{ab}^{w(ab)}=1$.
\item If $ab\in K\setminus A$, and $ab$ is a face of some 2-simplex $avb\in K$, then $g_{ab}$ satisfies the relation $g_{ab}^{w(ab)}=1$, since necessarily $av$ and $vb$ both lie in $A$ and hence $g_{ab}^{w(av)}=g_{vb}^{w(vb)}=1$.
\item If $ab\in K\setminus A$ and $ab$ is not a face of any 2-simplex of $K$, then $g_{ab}$ does not satisfy any relation.
\end{enumerate}

We observe that the generators $g_{ab}$ do not satisfy any other relation, other than the relations listed above.

Hence, 
\begin{equation*}
\begin{split}
\pi_1(K,w,A)&\cong\prod^*_{ab\in A}\langle g_{ab}\mid g_{ab}^{w(ab)}=1\rangle*\prod^*_{\substack{ab\in K\setminus A \\ \text{and $ab$ is a face of}\\ \text{some 2-simplex of $K$}}}\langle g_{ab}\mid g_{ab}^{w(ab)}=1\rangle*\\
&\prod^*_{\substack{ab\in K\setminus A \\ \text{and $ab$ is not a face of}\\ \text{any 2-simplex of $K$}}}\langle g_{ab}\rangle.
\end{split}
\end{equation*}
\end{proof}

\begin{cor}
\label{cor:exactly2}
Let $(K,w,A)$ be a weighted simplicial complex such that for each 2-simplex $avb\in K$, exactly 2 of the 1-simplices $ab$, $av$, $vb$ lie in the maximal tree $A$.

Let $w'$ be another weight function from the 1-simplices of $K$ to $\Z$ such that  for all 1-simplices $ab\in K$, either $w'(ab)=w(ab)$ or $w'(ab)=-w(ab)$.

Then, $\pi_1(K,w,A)\cong\pi_1(K,w',A)$.
\end{cor}
\begin{proof}
The proof follows from Theorem \ref{thm:exactly2formula} since the groups $\Z/w(ab)$ and $\Z/(-w(ab))$ are identical.
\end{proof}

\begin{cor}
\label{cor:pm1}
Let $(K,w,A)$ be a path-connected WSC such that for each 2-simplex $avb\in K$, exactly 2 of the 1-simplices $ab$, $av$, $vb$ lie in the maximal tree $A$. Suppose that either $w(\sigma)=1$ or $w(\sigma)=-1$ for all 1-simplices $\sigma\in K$.

Then, $\pi_1(K,w,A)$ is isomorphic to the usual fundamental group $\pi_1(K)$.
\end{cor}
\begin{proof}
Let $(K,w',A)$ be the WSC such that $w'(\sigma)\equiv 1$ for all 1-simplices $\sigma\in K$. By Theorem \ref{thm:allweights1}, we have 
\begin{equation}
\label{eq:coriso1}
\pi_1(K,w',A)\cong\pi_1(K).
\end{equation}

Since $w'(\sigma)=w(\sigma)$ or $w'(\sigma)=-w(\sigma)$, by Corollary \ref{cor:exactly2} we have
\begin{equation}
\label{eq:coriso2}
\pi_1(K,w,A)\cong\pi_1(K,w',A).
\end{equation}

Combining both equations \eqref{eq:coriso1} and \eqref{eq:coriso2} concludes the proof.
\end{proof}

\begin{remark}
By Corollary \ref{cor:pm1}, we can conclude that when the weights in Examples \ref{eg:circleK} and \ref{eg:circleL} are all $\pm 1$, then the respective weighted fundamental groups are isomorphic to the usual fundamental groups.
\end{remark}

In Section \ref{sec:weightedgraph}, we show that the weighted fundamental group of a weighted graph has a simple structure, being the free product of its usual fundamental group and the cyclic groups associated with its maximal tree (Corollary \ref{cor:pi1weightedgraph}). In general, this is not the case. In fact the usual fundamental group may not appear as a free factor of the weighted fundamental group. We illustrate this in the following example.

\begin{eg}
\label{eg:specialK}
Let $K$ be the simplicial complex (homotopy equivalent to a wedge of 2 circles $S^1\vee S^1$) shown in Figure \ref{fig:wedge2circle}, where the maximal tree $A$ is marked in bold. Suppose all the weights $w_{ij}$ are equal to 2.

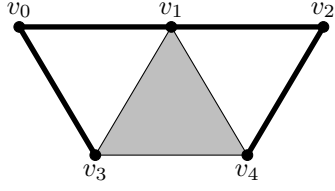
\begin{figure}[htbp]
\begin{center}
\begin{tikzpicture}
\fill[fill=lightgray]
(0,0)  
-- (2,0)
-- (1,1.7);
\filldraw 
(0,0) circle (2pt) node[align=center,below] {$v_3$}
(2,0) circle (2pt) node[align=center,below] {$v_4$}
(-1,1.7) circle (2pt) node[align=center,above] {$v_0$}  
(3,1.7) circle (2pt) node[align=center,above] {$v_2$}
(1,1.7) circle (2pt) node[align=center,above] {$v_1$};
\draw (0,0)--(2,0);
\draw (2,0)--(1,1.7);
\draw (1,1.7)--(0,0);
\draw[line width=2pt] (-1,1.7)--(1,1.7);
\draw[line width=2pt] (1,1.7)--(3,1.7);
\draw[line width=2pt] (0,0)--(-1,1.7);
\draw[line width=2pt] (2,0)--(3,1.7);
\end{tikzpicture}
\caption{The simplicial complex $K$ with maximal tree $A$ (marked in bold).}
\label{fig:wedge2circle}
\end{center}
\end{figure}

We can calculate that 
\begin{equation*}
\pi_1(K,w,A)\cong\Z/2*\Z/2*\Z/2*\Z/2*\langle a,b,c\mid a^2=b^2c^2\rangle.
\end{equation*}

We note that the usual fundamental group $\pi_1(K)\cong\Z*\Z$ does not appear as a free factor in $\pi_1(K,w,A)$.
\end{eg}


\section{Weighted Graphs}
\label{sec:weightedgraph}
Recall that in this paper, we let $K$ be a path-connected simplicial complex.
\begin{defn}
We define a \emph{graph} to be a 1-dimensional simplicial complex. A WSC $(K,w,A)$ is said to be a \emph{weighted graph} if $K$ is a graph.
\end{defn}

The following proposition is a weighted version of the classical result that the fundamental group of a graph is a free group (see \cite[p.~242]{hilton1960introduction}). 

\begin{prop}
\label{prop:productcyclic}
Let $(K,w,A)$ be a weighted graph. Then $\pi_1(K,w,A)$ is a free product of cyclic groups.

To be precise, we have 
\begin{equation*}
\pi_1(K,w,A)\cong\prod^*_{ab\in K\setminus A}\Z*\prod^*_{ab\in A}\Z/w(ab),
\end{equation*}
where $ab$ denotes a 1-simplex.
\end{prop}
\begin{proof}
The proof follows from Theorem \ref{thm:exactly2formula}. Since a graph does not have any 2-simplices, the condition of Theorem \ref{thm:exactly2formula} is trivially satisfied. Any 1-simplex $ab\in K\setminus A$ is not a face of any 2-simplex of $K$, since there are no 2-simplices in $K$.
\end{proof}

The following corollary shows that when all weights are equal, the weighted fundamental group of a weighted graph is independent of the choice of maximal tree.

\begin{cor}
\label{cor:indeptreegraph}
Let $(K,w,A)$ be a weighted graph. If $w(\sigma)\equiv a\in\Z$ for all 1-simplices $\sigma\in K$, then $\pi_1(K,w,A)\cong\pi_1(K,w,B)$ for any maximal tree $B$ in $K$.
\end{cor}
\begin{proof}
Let $V$ be the number of vertices of $K$ and let $E$ be the number of 1-simplices (edges) of $K$. Then, each maximal tree has exactly $V-1$ edges. Hence, by Proposition \ref{prop:productcyclic}, $\pi_1(K,w,A)$ and $\pi_1(K,w,B)$ are both isomorphic to a free product of $E-(V-1)$ copies of $\Z$ and $V-1$ copies of $\Z/a$.
\end{proof}

As a corollary, we also have the following decomposition result that relates the weighted fundamental group $\pi_1(K,w,A)$ of a weighted graph to the usual fundamental group $\pi_1(K)$.

\begin{cor}
\label{cor:pi1weightedgraph}
Let $(K,w,A)$ be a weighted graph. Then
\begin{equation*}
\pi_1(K,w,A)\cong\pi_1(K)*\prod^*_{ab\in A}\Z/w(ab),
\end{equation*}
where $ab$ denotes a 1-simplex.
\end{cor}
\begin{proof}
For a graph, the fundamental group $\pi_1(K)$ is a free group generated by the generators $g_{ab}$, one for each 1-simplex of $K\setminus A$. That is, $\pi_1(K)\cong\prod^*_{ab\in K\setminus A}\Z$. Hence, the result follows from Proposition \ref{prop:productcyclic}.
\end{proof}

\begin{prop}
Let $G$ be a free product of cyclic groups. There exists a weighted graph $(K,w,A)$ with $\pi_1(K,w,A)\cong G$.
\end{prop}
\begin{proof}
Write 
\begin{equation*}
G\cong\prod^*_{i\in I}\langle g_i\mid g_i^{m_i}=1\rangle\cong\prod^*_{i\in I}\Z/m_i,
\end{equation*}
where $m_i\in\Z$ and $I$ is some index set.

For each generator $g_i$ of $G$, take a 1-simplex $L_i$ and let $K=\bigvee_{i\in I} L_i$ be the wedge sum of all such $L_i$. Choose $A=K$, where we observe that $A$ is a maximal tree of $K$. Let $w(L_i)=m_i$. Then $\pi_1(K,w,A)$ is generated by the letters $g_{L_i}$ with relations $g_{L_i}^{w(L_i)}=1$.

Hence,
\begin{equation*}
\begin{split}
\pi_1(K,w,A)&\cong\prod^*_{i\in I}\langle g_{L_i}\mid g_{L_i}^{w(L_i)}=1\rangle\\
&\cong\prod^*_{i\in I}\langle g_i\mid g_i^{m_i}=1\rangle\\
&\cong G.
\end{split}
\end{equation*}
\end{proof}

\section{Weighted van Kampen Theorem}
\label{sec:vankampen}
In this section, we generalize van Kampen's theorem for the case of weighted simplicial complexes. We recall that a weighted simplicial complex $(K,w,A)$ is a triple consisting of a path-connected simplicial complex $K$, a weight function $w$ and a maximal tree $A$ of $K$.

\begin{defn}[Weighted subcomplex]
Let $(K,w_K,A)$ and $(L,w_L,B)$ be weighted simplicial complexes such that:
\begin{enumerate}
\item $K\subseteq L$,
\item $w_K$ is equal to the restriction of $w_L$ to $K$, i.e.\ $w_K=w_L|_K$,
\item $A\subseteq B$, and
\item the ordering of vertices in $K$ is preserved in $L$, i.e.\ if $a<b$ in $K$, then $a<b$ in $L$. 
\end{enumerate}

We call $(K,w_K,A)$ a \emph{weighted subcomplex} of $(L,w_L,B)$.
\end{defn}

\begin{defn}
Let $(K,w_K,A)$ be a weighted subcomplex of $(L,w_L,B)$. Let $i: K\to L$ be the inclusion map $i(\sigma)=\sigma$. We call $i$ an inclusion map between the WSCs $(K,w_K,A)$ and $(L,w_L,B)$. 
\end{defn}

We first state an algebraic lemma that will be used subsequently.
\begin{lemma}
\label{lemma:genextend}
Let $G=\langle S\mid R\rangle$ be a group with generators $S=\{g_i\}$ and relations $R=\{r_j=1\}$. Let $H$ be another group, and let $\phi$ be a function that assigns to each $g_i$ a value $\phi(g_i)\in H$.

Then the function $\phi$ extends to a homomorphism $\widetilde{\phi}:G\to H$ iff for each relation $r_j=g_{i_1}^{k_1}\dots g_{i_m}^{k_m}=1$, where $k_1,\dots k_m=\pm 1$, we have that $\phi(g_{i_1})^{k_1}\dots\phi(g_{i_m})^{k_m}=1$.
\qed
\end{lemma}

\begin{prop}
Let $(K,w_K,A)$ be a weighted subcomplex of $(L,w_L,B)$. Let $i: K\to L$ be an inclusion map between the WSCs $(K,w_K,A)$ and $(L,w_L,B)$. Then the map $i$ induces a homomorphism
\begin{equation*}
i_*: \pi_1(K,w_K,A)\to\pi_1(L,w_L,B).
\end{equation*}

We call $i_*$ the \emph{induced homomorphism} of $i$.
\end{prop}
\begin{proof}
Let $\{g_{ab}\mid ab\in K\}$ be the generators of $\pi_1(K,w_K,A)$ with defining relations given in Definition \ref{defn:relations}. Similarly, let $\{h_{cd}\mid cd\in L^1\}$ be the generators of $\pi_1(L,w_L,B)$ with similar defining relations as given in Definition \ref{defn:relations}.

We define a map $\phi$ from the generators of $\pi_1(K,w_K,A)$ to that of $\pi_1(L,w_L,B)$ by 
\begin{equation*}
\phi(g_{ab})=h_{ab}.
\end{equation*}

Suppose $ab\in A$, then 
\begin{equation*}
\phi(g_{ab})^{w(ab)}=h_{ab}^{w(ab)}=1
\end{equation*}
since $ab\in A\subseteq B$.

Suppose $a<v<b$ and $avb$ is a 2-simplex of $K\subseteq L$, then
\begin{equation*}
\begin{split}
\phi(g_{ab})^{w(ab)}&=h_{ab}^{w(ab)}\\
&=h_{av}^{w(av)}h_{vb}^{w(vb)}\qquad\text{(since $a<v<b$ in $L$ and $avb$ is also a 2-simplex in $L$)}\\
&=\phi(g_{av})^{w(av)}\phi(g_{vb})^{w(vb)}.
\end{split}
\end{equation*}

Hence by Lemma \ref{lemma:genextend}, $\phi$ extends to a homomorphism
\begin{equation*}
i_*: \pi_1(K,w_K,A)\to\pi_1(L,w_L,B).
\end{equation*}
\end{proof}

\begin{lemma}
\label{lemma:subtrees}
Let $(L,w_L,B)$ be a path-connected WSC. Let $(K_0,w_{K_0}, A_0)$, $(K_1,w_{K_1},A_1)$, $(K_2, w_{K_2}, A_2)$ be path-connected weighted subcomplexes of $(L,w_L,B)$ such that:
\begin{enumerate}
\item $K_1\cup K_2=L$,
\item $K_1\cap K_2=K_0$, and
\item $(K_0,w_{K_0},A_0)$ is a weighted subcomplex of both $(K_1,w_{K_1},A_1)$ and $(K_2,w_{K_2},A_2)$.
\end{enumerate}

Then,
\begin{equation}
\label{eq:uniontree}
B=A_1\cup A_2
\end{equation} 
and
\begin{equation}
\label{eq:intersecttree}
A_1\cap A_2=A_0.
\end{equation}
\end{lemma}
\begin{proof}
We first prove \eqref{eq:uniontree}. By definition of weighted subcomplex, we have $A_1\subseteq B$ and $A_2\subseteq B$. Hence, $A_1\cup A_2\subseteq B$. Let $\sigma$ be a 0-simplex in $B\subseteq L$. Since $K_1\cup K_2=L$, this implies that $\sigma\in K_1$ or $\sigma\in K_2$. Since $A_1$ is a maximal tree of $K_1$, $A_1$ contains all vertices of $K_1$. Similarly, $A_2$ contains all vertices of $K_2$. Hence, $\sigma\in A_1\cup A_2$.

Let $\tau$ be a 1-simplex in $B\subseteq L$. Similarly, since $K_1\cup K_2=L$, this implies that $\tau\in K_1$ or $\tau\in K_2$. Suppose $\tau\in K_1$. Then, we claim that $\tau\in A_1$. Otherwise $\tau\notin A_1$ implies that $\{\tau\}\cup A_1\subseteq B$ contains a cycle which contradicts the fact that $B$ is a tree. (We have used the property that adding one edge in $K_1\setminus A_1$ to the maximal tree $A_1$ will create a cycle.) Similarly, if $\tau\in K_2$ then necessarily $\tau\in A_2$. Hence, $\tau\in A_1\cup A_2$. We have shown that $B\subseteq A_1\cup A_2$. Therefore, $B=A_1\cup A_2$.

To prove \eqref{eq:intersecttree}, note that by definition of weighted subcomplex, we have $A_0\subseteq A_1$ and $A_0\subseteq A_2$. Thus, $A_0\subseteq A_1\cap A_2$. Let $\sigma$ be a 0-simplex in $A_1\cap A_2\subseteq K_1\cap K_2=K_0$. Since the maximal tree $A_0$ contains all vertices of $K_0$, hence $\sigma\in A_0$.

Let $\tau$ be a 1-simplex in $A_1\cap A_2\subseteq K_1\cap K_2=K_0$. Suppose to the contrary $\tau\notin A_0$. Then $\{\tau\}\cup A_0\subseteq A_1$ contains a cycle which contradicts the fact that $A_1$ is a tree. Hence, $\tau\in A_0$. We have shown that $A_1\cap A_2\subseteq A_0$. This completes the proof of \eqref{eq:intersecttree}.
\end{proof}

We now prove the main theorem of this section, the Weighted van Kampen Theorem. It is the weighted version of van Kampen's theorem for simplicial complexes (cf.\ \cite[p.~243]{hilton1960introduction}).

\begin{theorem}[Weighted van Kampen Theorem]
\label{thm:vankampen}
Let $(L,w_L,B)$ be a path-connected WSC. Let $(K_0,w_{K_0}, A_0)$, $(K_1,w_{K_1},A_1)$, $(K_2, w_{K_2}, A_2)$ be path-connected weighted subcomplexes of $(L,w_L,B)$ such that:
\begin{enumerate}
\item $K_1\cup K_2=L$,
\item $K_1\cap K_2=K_0$, and
\item $(K_0,w_{K_0},A_0)$ is a weighted subcomplex of both $(K_1,w_{K_1},A_1)$ and $(K_2,w_{K_2},A_2)$.
\end{enumerate}

Let $i_1: K_0\to K_1$, $i_2: K_0\to K_2$ be the inclusion maps between $(K_0,w_{K_0},A_0)$ and the WSCs $(K_1,w_{K_1},A_1)$ and $(K_2,w_{K_2},A_2)$ respectively.

Then,
\begin{equation*}
\pi_1(L,w_L,B)=\pi_1(K_1,w_{K_1},A_1)\coprod_{\pi_1(K_0,w_{K_0},A_0)}\pi_1(K_2,w_{K_2},A_2),
\end{equation*}
the free product with amalgamation of $\pi_1(K_1,w_{K_1},A_1)$ and $\pi_1(K_2,w_{K_2},A_2)$ with respect to the induced homomorphisms
\begin{align*}
{i_1}_*&: \pi_1(K_0,w_{K_0},A_0)\to\pi_1(K_1,w_{K_1},A_1)\\
{i_2}_*&: \pi_1(K_0,w_{K_0},A_0)\to\pi_1(K_2,w_{K_2},A_2).
\end{align*}

In other words, $\pi_1(L,w_L,B)$ is obtained from the free product of $\pi_1(K_1,w_{K_1},A_1)$ and $\pi_1(K_2,w_{K_2},A_2)$ by adding the relations ${i_1}_*(\alpha)={i_2}_*(\alpha)$ for all $\alpha\in\pi_1(K_0,w_{K_0},A_0)$.
\end{theorem}
\begin{proof}
The generators $\{g_{ab}\}$ of $\pi_1(L,w_L,B)$ fall into 6 classes (depending on where the 1-simplex $ab$ is):
\begin{enumerate}
\item $ab\in K_0\cap A_0=A_0$,
\item $ab\in K_0\setminus A_0$,
\item $ab\in (K_1\setminus K_0)\cap A_1$,
\item $ab\in K_1\setminus (K_0\cup A_1)=(K_1\setminus K_0)\cap (K_1\setminus A_1)$,
\item $ab\in (K_2\setminus K_0)\cap A_2$,
\item $ab\in K_2\setminus (K_0\cup A_2)=(K_2\setminus K_0)\cap (K_2\setminus A_2)$.
\end{enumerate}

By construction, the above 6 classes of generators are mutually exclusive. By the condition $K_1\cup K_2=L$, the above 6 classes covers all possible cases.

The relations of $\pi_1(L,w_L,B)$ similarly fall into 6 classes:
\begin{enumerate}
\item $g_{ab}^{w(ab)}=1$ if $ab\in A_0$,
\item $g_{ab}^{w(ab)}=g_{av}^{w(av)}g_{vb}^{w(vb)}$ if $a<v<b$ and $avb\in K_0$,
\item $g_{ab}^{w(ab)}=1$ if $ab\in(K_1\setminus K_0)\cap A_1$,
\item $g_{ab}^{w(ab)}=g_{av}^{w(av)}g_{vb}^{w(vb)}$ if $a<v<b$ and $avb\in K_1\setminus K_0$,
\item $g_{ab}^{w(ab)}=1$ if $ab\in (K_2\setminus K_0)\cap A_2$,
\item $g_{ab}^{w(ab)}=g_{av}^{w(av)}g_{vb}^{w(vb)}$ if $a<v<b$ and $avb\in K_2\setminus K_0$.
\end{enumerate}

By Lemma \ref{lemma:subtrees}, $B=A_1\cup A_2$ and $A_1\cap A_2=A_0$. Hence any $ab\in B$ will fall into exactly one class of relations of type (1), (3) or (5). Since $K_1\cup K_2=L$, any $avb\in L$ will fall into exactly one class of relations of type (2), (4) or (6).

The generators and relations of types (1), (2), (3), (4) define $\pi_1(K_1,w_{K_1}, A_1)$ while those of types (1), (2), (5), (6) define $\pi_1(K_2,W_{K_2},A_2)$. Now, for each generator $g_{ab}$ of type (1) or (2), i.e.\ when $ab\in K_0$, we consider two new generators $g'_{ab}$ and $g''_{ab}$.

Then $\pi_1(L,w_L,B)$ is the group generated by $\{g_{ab}\}$, $\{g'_{ab}\}$, $\{g''_{ab}\}$ subject to 9 types of relations as described:
\begin{itemize}
\item The first 4 types correspond to the relations of types (1), (2), (3), (4), but replacing $g_{cd}$ with $g'_{cd}$ whenever $cd\in K_0$.
\item The next 4 types correspond to the relations of types (1), (2), (5), (6), but replacing $g_{cd}$ with $g''_{cd}$ whenever $cd\in K_0$.
\item The final 9th type of relation is $g'_{ab}=g''_{ab}$ for all $ab\in K_0$.
\end{itemize}

The group with the generators $\{g_{ab}\}$, $\{g'_{ab}\}$, $\{g''_{ab}\}$ and the first 8 types of relations is the free product of $\pi_1(K_1,w_{K_1}, A_1)$ and $\pi_1(K_2,w_{K_2},A_2)$. The 9th type of relation corresponds to the amalgamation that identifies two elements of $\pi_1(K_1,w_{K_1},A_1)$ and $\pi_1(K_2,w_{K_2},A_2)$ that arise from the same element of $\pi_1(K_0,w_{K_0},A_0)$.
\end{proof}

We can also restate the weighted van Kampen theorem in terms of a commutative pushout diagram (cf.\ \cite[p.~426]{munkres2000topology}).

\begin{theorem}
\label{thm:catvankampen}
Let $(L,w_L,B)$, $(K_0,w_{K_0},A_0)$, $(K_1,w_{K_1},A_1)$, $(K_2,w_{K_2},A_2)$ be path-connected WSCs satisfying the same conditions as in Theorem \ref{thm:vankampen}.

Then the following diagram commutes:

\begin{center}
\begin{tikzcd}
\pi_1(L,w_L,B)
& & \\
&\pi_1(K_1,w_{K_1},A_1)\displaystyle\coprod_{\pi_1(K_0,w_{K_0},A_0)}\pi_1(K_2,w_{K_2},A_2)
\arrow{ul}{k}
& \pi_1(K_2,w_{K_2},A_2)\arrow{l}{\phi_2}\arrow[bend right]{ull}{{j_2}_*}\\
&\pi_1(K_1,w_{K_1},A_1)\arrow{u}{\phi_1}\arrow[bend left]{uul}{{j_1}_*} & \pi_1(K_0,w_{K_0},A_0)\arrow{l}{{i_1}_*}\arrow{u}{{i_2}_*}
\end{tikzcd}
\end{center}

In the diagram, ${i_1}_*$, ${i_2}_*$, ${j_1}_*$, ${j_2}_*$ are homomorphisms induced by the respective inclusion maps. The maps $\phi_1$, $\phi_2$ are inclusions into the free product followed by projection onto the quotient. The morphism $k$ is an isomorphism.
\qed
\end{theorem}

We illustrate the weighted van Kampen theorem in the following example.

\begin{eg}
\begin{figure}[htbp]
\begin{subfigure}{1\textwidth}
\centering
\begin{tikzpicture}[scale=0.8]
\fill[fill=lightgray]
(2,0)  
-- (3,-1.7)
-- (4,0);
\filldraw 
(0,0) circle (2pt) node[below] {$v_1$}
(2,0) circle (2pt) node[below] {} 
(1.8,0) circle (0pt) node[below] {$v_2$}
(1,1.7) circle (2pt) node[above] {$v_0$}
(3,-1.7) circle (2pt) node[below] {$v_3$}
(4,0) circle (2pt) node[below] {} 
(4.2,0) circle (0pt) node[below] {$v_4$}
(5,1.7) circle (2pt) node[above] {$v_5$}
(6,0) circle (2pt) node[below] {$v_6$};
\draw[line width=2pt] (0,0)--(2,0);
\draw (2,0)--(1,1.7);
\draw[line width=2pt] (1,1.7)--(0,0);
\draw[line width=2pt] (2,0)--(3,-1.7);
\draw[line width=2pt] (3,-1.7)--(4,0);
\draw (2,0)--(4,0);
\draw[line width=2pt] (4,0)--(6,0);
\draw[line width=2pt] (5,1.7)--(6,0);
\draw (4,0)--(5,1.7);
\end{tikzpicture}
\caption{$(L,w_L,B)$}
\end{subfigure}
\begin{subfigure}{0.49\textwidth}
\centering
\begin{tikzpicture}[scale=0.8]
\fill[fill=lightgray]
(2,0)  
-- (3,-1.7)
-- (4,0);
\filldraw 
(0,0) circle (2pt) node[below] {$v_1$}
(2,0) circle (2pt) node[below] {} 
(1.8,0) circle (0pt) node[below] {$v_2$}
(1,1.7) circle (2pt) node[above] {$v_0$}
(3,-1.7) circle (2pt) node[below] {$v_3$}
(4,0) circle (2pt) node[below] {} 
(4.2,0) circle (0pt) node[below] {$v_4$};
\draw[line width=2pt] (0,0)--(2,0);
\draw (2,0)--(1,1.7);
\draw[line width=2pt] (1,1.7)--(0,0);
\draw[line width=2pt] (2,0)--(3,-1.7);
\draw[line width=2pt] (3,-1.7)--(4,0);
\draw (2,0)--(4,0);
\end{tikzpicture}
\caption{$(K_1,w_{K_1},A_1)$}
\end{subfigure}
\begin{subfigure}{0.49\textwidth}
\centering
\begin{tikzpicture}[scale=0.8]
\fill[fill=lightgray]
(2,0)  
-- (3,-1.7)
-- (4,0);
\filldraw 
(2,0) circle (2pt) node[below] {} 
(1.8,0) circle (0pt) node[below] {$v_2$}
(3,-1.7) circle (2pt) node[below] {$v_3$}
(4,0) circle (2pt) node[below] {} 
(4.2,0) circle (0pt) node[below] {$v_4$}
(5,1.7) circle (2pt) node[above] {$v_5$}
(6,0) circle (2pt) node[below] {$v_6$};
\draw[line width=2pt] (2,0)--(3,-1.7);
\draw[line width=2pt] (3,-1.7)--(4,0);
\draw (2,0)--(4,0);
\draw[line width=2pt] (4,0)--(6,0);
\draw[line width=2pt] (5,1.7)--(6,0);
\draw (4,0)--(5,1.7);\end{tikzpicture}
\caption{$(K_2,w_{K_2},A_2)$}
\end{subfigure}
\begin{subfigure}{0.49\textwidth}
\centering
\begin{tikzpicture}[scale=0.8]
\fill[fill=lightgray]
(2,0)  
-- (3,-1.7)
-- (4,0);
\filldraw 
(2,0) circle (2pt) node[below] {} 
(1.8,0) circle (0pt) node[below] {$v_2$}
(3,-1.7) circle (2pt) node[below] {$v_3$}
(4,0) circle (2pt) node[below] {} 
(4.2,0) circle (0pt) node[below] {$v_4$};
\draw[line width=2pt] (2,0)--(3,-1.7);
\draw[line width=2pt] (3,-1.7)--(4,0);
\draw (2,0)--(4,0);
\end{tikzpicture}
\caption{$(K_0,w_{K_0},A_0)$}
\end{subfigure}
\caption{The 4 WSCs with their respective maximal trees $B, A_1, A_2, A_0$ marked in bold. The 4 WSCs are chosen such that they satisfy the conditions of the weighted van Kampen theorem (Theorem \ref{thm:vankampen}).}
\label{fig:vankampen}
\end{figure}
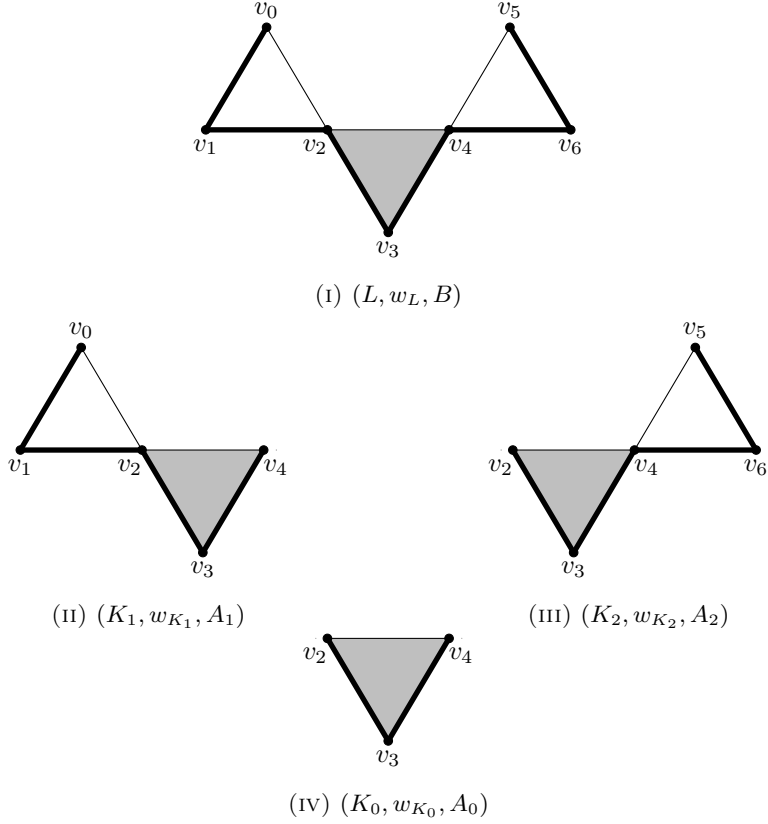

Consider the path-connected WSCs in Figure \ref{fig:vankampen}, with their respective maximal trees marked in bold. The 4 WSCs are chosen such that they satisfy the conditions of the weighted van Kampen theorem (Theorem \ref{thm:vankampen}). We write $w_{ij}$ for $w_L(v_iv_j)$ and $g_{ij}$ for $g_{v_i}g_{v_j}$.

Calculating $\pi_1(L,w_L,B)$ from the definition (Definition \ref{defn:relations}), we have that the generators of $\pi_1(L,w_L,B)$ are

\begin{equation*}
g_{01}, g_{02}, g_{12}, g_{23}, g_{24}, g_{34}, g_{45}, g_{46}, g_{56}
\end{equation*}

with defining relations given by:
\begin{enumerate}
\item $g_{01}^{w_{01}}=g_{12}^{w_{12}}=g_{23}^{w_{23}}=g_{34}^{w_{34}}=g_{46}^{w_{46}}=g_{56}^{w_{56}}=1$,
\item $g_{24}^{w_{24}}=g_{23}^{w_{23}}g_{34}^{w_{34}}=1$.
\end{enumerate}

Hence,
\begin{equation}
\label{eq:vankampendef}
\pi_1(L,w_L,B)\cong\Z*\Z*\Z/w_{01}*\Z/w_{12}*\Z/w_{23}*\Z/w_{34}*\Z/w_{46}*\Z/w_{56}*\Z/w_{24}.
\end{equation}

Note that in particular, if all the weights $w_{01}, w_{12}, w_{23}, w_{34}, w_{46}, w_{56}, w_{24}$ are equal to $\pm 1$, we have $\pi_1(L,w_L,B)\cong\Z*\Z$. This corresponds to the usual fundamental group $\pi_1(L)\cong\Z*\Z$ since $L$ is homotopy equivalent to a wedge of two circles.

Similarly, by definition we can compute that 
\begin{equation*}
\pi_1(K_1,w_{K_1},A_1)=\langle g_{01}, g_{02}, g_{12}, g_{23}, g_{24}, g_{34}\mid g_{01}^{w_{01}}=g_{12}^{w_{12}}=g_{23}^{w_{23}}=g_{34}^{w_{34}}=1, g_{24}^{w_{24}}=1\rangle
\end{equation*}

and
\begin{equation*}
\pi_1(K_2,w_{K_2},A_2)=\langle g_{23}, g_{24}, g_{34}, g_{45}, g_{46}, g_{56}\mid g_{23}^{w_{23}}=g_{34}^{w_{34}}=g_{46}^{w_{46}}=g_{56}^{w_{56}}=1, g_{24}^{w_{24}}=1\rangle.
\end{equation*}

By the weighted van Kampen theorem (Theorem \ref{thm:vankampen}), we have that
\begin{equation}
\label{eq:vankampeneg}
\begin{split}
\pi_1(L,w_L,B)&=\pi_1(K_1,w_{K_1},A_1)\coprod_{\pi_1(K_0,w_{K_0},A_0)}\pi_1(K_2,w_{K_2},A_2)\\
&=\langle g_{01}, g_{02}, g_{12}, g_{23}, g_{24}, g_{34}, g_{45}, g_{46}, g_{56}\mid\\
&\qquad g_{01}^{w_{01}}=g_{12}^{w_{12}}=g_{23}^{w_{23}}=g_{34}^{w_{34}}=1, g_{24}^{w_{24}}=1, g_{46}^{w_{46}}=g_{56}^{w_{56}}=1\rangle\\
&\cong\Z*\Z*\Z/w_{01}*\Z/w_{12}*\Z/w_{23}*\Z/w_{34}*\Z/w_{24}*\Z/w_{46}*\Z/w_{56}.
\end{split}
\end{equation}

Note that \eqref{eq:vankampeneg} agrees with our previous computation of $\pi_1(L,w_L,B)$ using the definition \eqref{eq:vankampendef}.
\end{eg}
\section{Weighted Homology Group and Abelianization of Weighted Fundamental Group}
In the paper \cite{Dawson1990}, Robert Dawson introduced the weighted homology of simplicial complexes. Subsequently, S.\ Ren, C.\ Wu and J.\ Wu \cite{ren2018weighted,ren2017further} generalized the definition of weighted homology, and studied the theory of weighted persistent homology. In \cite{wu2018weighted}, C.\ Wu, S.\ Ren, J.\ Wu and K.\ Xia studied the weighted (co)homology of simplicial complexes by considering the $\phi$-weighted  (co)boundary operator. The approach in \cite{wu2018weighted} generalizes the previous definitions of weighted homology in \cite{Dawson1990,ren2018weighted,ren2017further}. Weighted (co)homology has various applications in data analysis \cite{ren2018weighted}, biochemistry \cite[p.~18]{wu2018weighted} and network motifs \cite[p.~22]{wu2018weighted}.

As an application of our definition of the weighted fundamental group, we show that the weighted fundamental group is able to distinguish between weighted simplicial complexes with the same first weighted homology group. Hence, it can be said that the weighted fundamental group contains more information than the first weighted homology group.

\begin{eg}
\label{eg:sameweightedhom}
Consider the simplicial complex $K$ in Figure \ref{fig:circle}. By the computations in Example \ref{eg:circleK}, the weighted fundamental group is $\pi_1(K,w,A)\cong\Z*(\Z/w_{01})*(\Z/w_{12})$. We can observe that different choices of weights $w_{01}$, $w_{12}$ lead to different weighted fundamental groups.

Let $\partial_1^w$ be the 1st weighted boundary map (see \cite[Definition~4.4]{ren2018weighted}). Suppose $w_{01}$, $w_{02}$, $w_{12}$ are nonzero. We may define the weights $w([v_0])$, $w([v_1])$, $w([v_2])$ to be 1. Then, we have
\begin{align*}
\partial_1^w([v_0,v_1])&=w_{01}[v_1]-w_{01}[v_0],\\
\partial_1^w([v_1,v_2])&=w_{12}[v_2]-w_{12}[v_1],\\
\partial_1^w([v_0,v_2])&=w_{02}[v_2]-w_{02}[v_0].
\end{align*}

We can verify that $\ker\partial_1^w\cong\Z$, generated by $w_{02}w_{12}[v_0,v_1]+w_{01}w_{02}[v_1,v_2]-w_{01}w_{12}[v_0,v_2]$. Since there is no 2-simplex in $K$, hence the weighted homology is $H_1(K,w)=\ker\partial_1^w\cong\Z$.

Consider 2 different WSCs $(K,w,A)$ and $(K,w',A)$ where $w_{01}=2$, $w_{02}=1$, $w_{12}=4$, while $w'_{01}=w'_{02}=w'_{12}=1$. We have $\pi_1(K,w,A)\cong\Z*(\Z/2)*(\Z/4)$ but $\pi_1(K,w',A)\cong\Z$.

Meanwhile, both $H_1(K,w)$ and $H_1(K,w')$ are isomorphic to $\Z$. (The difference in weighted homology in this case is reflected in the 0-th homology group, where $H_0(K,w)\cong\Z\oplus\Z_2$ while $H_0(K,w')\cong\Z$.)
\end{eg}

\begin{remark}
\label{remark:poincarefail}
A classical theorem of Poincar\'{e} (see \cite[p.~166]{Hatcher2002}) states that if $X$ is a path-connected space, then $H_1(X;\Z)$ is the abelianization of $\pi_1(X)$. From Example \ref{eg:sameweightedhom}, we see that Poincar\'{e}'s Theorem fails in the weighted case, since the abelianization of $\pi_1(K,w,A)$ is isomorphic to $\Z\oplus\Z/2\oplus\Z/4$ but $H_1(K,w)\cong\Z$.
\end{remark}

\subsection{Abelianization of Weighted Fundamental Group}
\label{subsec:abelian}
In Remark \ref{remark:poincarefail}, it is shown that Poincar\'{e}'s theorem fails in the weighted case. Hence, it is natural to study the abelianization of the weighted fundamental group.

\begin{defn}[{\cite[p.~58]{bergman2015invitation}}]
For a group $G$, we denote its \emph{abelianization} by
\begin{equation*}
\Ab(G):=G/[G,G],
\end{equation*}
where $[G,G]$ is the commutator subgroup.
\end{defn}

We state the following basic lemma, which will be used subsequently.

\begin{lemma}
Let $G_1$ and $G_2$ be groups. Then we have
\begin{equation*}
\Ab(G_1*G_2)\cong\Ab(G_1)\oplus\Ab(G_2).
\end{equation*}
\qed
\end{lemma}

Since the free product is associative, by induction we can obtain the following corollary.

\begin{cor}
\label{cor:abfreeprods}
Let $G_1,\dots G_n$ be groups. Then
\begin{equation*}
\Ab(G_1*\dots*G_n)\cong\bigoplus_{i=1}^n \Ab(G_i).
\end{equation*}
\qed
\end{cor}

\begin{prop}
\label{prop:abexactly2}
Let $(K,w,A)$ be a weighted simplicial complex such that for each 2-simplex $avb\in K$, exactly 2 of the 1-simplices $ab$, $av$, $vb$ lie in the maximal tree $A$.

Then,
\begin{equation*}
\Ab(\pi_1(K,w,A))\cong\bigoplus_{ab\in A}\Z/w(ab)\oplus\bigoplus_{\substack{ab\in K\setminus A \\ \text{and $ab$ is a face of}\\ \text{some 2-simplex of $K$}}}\Z/w(ab)\oplus\bigoplus_{\substack{ab\in K\setminus A \\ \text{and $ab$ is not a face of}\\ \text{any 2-simplex of $K$}}}\Z.
\end{equation*}
\end{prop}
\begin{proof}
The proof is obtained by applying Corollary \ref{cor:abfreeprods} to Theorem \ref{thm:exactly2formula}.
\end{proof}

\begin{prop}
\label{prop:abgraph}
Let $(K,w,A)$ be a weighted graph. Then
\begin{equation*}
\Ab(\pi_1(K,w,A))\cong H_1(K;\Z)\oplus \bigoplus_{ab\in A}\Z/w(ab),
\end{equation*}
where $ab$ denotes a 1-simplex.
\end{prop}
\begin{proof}
We apply Corollary \ref{cor:abfreeprods} and Poincar\'{e}'s theorem $\Ab(\pi_1(K))\cong H_1(K;\Z)$ to Corollary \ref{cor:pi1weightedgraph}. The result then follows since $\Ab(\Z/w(ab))\cong\Z/w(ab)$.
\end{proof}

We also calculate an example that is not covered by Propositions \ref{prop:abexactly2} and \ref{prop:abgraph}.

\begin{eg}
Consider Example \ref{eg:specialK}, where $K$ is the simplicial complex shown in Figure \ref{fig:wedge2circle}. We can compute that
\begin{equation*}
\begin{split}
\Ab(\pi_1(K,w,A))&\cong\bigoplus_{i=1}^4\Z/2\oplus\Ab(\langle a,b,a^{-1}bc\mid a^{-2}b^2c^2=1\rangle)\\
&\cong\bigoplus_{i=1}^5\Z/2\oplus\Z\oplus\Z.
\end{split}
\end{equation*}
\end{eg}

\section{Lower Central Series}
\label{sec:lowercs}
In \cite{gaglione1975factor}, Anthony M.\ Gaglione studied the factor groups of the lower central series for groups $G$ that are free products of finitely generated abelian groups. In the context of Theorem \ref{thm:exactly2formula} and weighted graphs (Proposition \ref{prop:productcyclic}), we recall that $\pi_1(K,w,A)$ is a free product of cyclic groups. Hence we may apply Gaglione's results to study the lower central series for $G=\pi_1(K,w,A)$ for these two special cases. In general, it is considered very difficult to describe the factor groups of the lower central series of an arbitrary group (cf.\ \cite{waldinger1970lower}).

\begin{defn}[{\cite[p.~160]{hall2018theory}}]
Let $G$ be a group. We define the following subgroups inductively:
\begin{equation*}
\begin{split}
\gamma_1(G)&=G,\\
\gamma_2(G)&=[G,G],\\
\gamma_{k+1}(G)&=[\gamma_k(G),G].
\end{split}
\end{equation*}

The series
\begin{equation*}
G=\gamma_1(G)\trianglerighteq \gamma_2(G)\trianglerighteq \gamma_3(G)\trianglerighteq \dots
\end{equation*}
is called the \emph{lower central series} of $G$.
\end{defn}

\begin{remark}
The abelianization of $G$, discussed in Section \ref{subsec:abelian}, is the quotient group $\gamma_1(G)/\gamma_2(G)$.
\end{remark}

For the rest of this section, we let $(K,w,A)$ be a weighted simplicial complex satisfying the conditions of Theorem \ref{thm:exactly2formula} (which includes the case of weighted graphs). To be precise, we let $(K,w,A)$ be a weighted simplicial complex such that for each 2-simplex $avb\in K$, exactly 2 of the 1-simplices $ab$, $av$, $vb$ lie in the maximal tree $A$. Then $G=\pi_1(K,w,A)$ is a free product of cyclic groups. In particular, $G$ is a free product of finitely generated abelian groups.

Thus, following the notation of \cite[p.~173]{gaglione1975factor}, we may write

\begin{equation}
\label{eq:g1gs}
G=G(1)*G(2)*\dots*G(s)
\end{equation}

where

\begin{equation}
\label{eq:cgen}
G(i)=\langle c_{n_{i-1}+1}\rangle\times\langle c_{n_{i-1}+2}\rangle\times\dots\times\langle c_{n_i}\rangle
\end{equation}

such that any generator $c_k$ $(0=n_0<k\leq n_s=r)$ has either infinite order or order a power of a prime $p(k)$.

\begin{defn}[{\cite[p.~230]{waldinger1970lower},\cite[p.~166]{hall2018theory}}]
The \emph{basic commutators} of dimension one are the free generators of the free group $F=\langle c_1,c_2,\dots,c_r\rangle$. We order the generators by $c_1<c_2<\dots<c_r$. We denote the \emph{dimension} of an element $a\in F$ by $D(a)$.

We then define basic commutators of dimension $n$ inductively as follows. The basic commutators of dimension $n$ are $c_m=[c_i,c_j]$ where $c_i$ and $c_j$ are basic commutators such that
\begin{enumerate}
\item $D(c_i)+D(c_j)=n$,
\item $c_i>c_j$, and
\item if $c_i=[c_s,c_t]$, then $c_j\geq c_t$.
\end{enumerate}
\end{defn}

In \cite[p.~174]{gaglione1975factor}, basic commutators are further classified into 4 classes, namely $G$-simple, $F$-simple, $I$-simple and $J$-simple. We will omit their precise definitions in this paper as they are not used in the subsequent discussions.

The following is the main theorem for this section.

\begin{theorem}[{cf.\ \cite[p.~175]{gaglione1975factor}}]
\label{thm:gaglione}
Let $G=\pi_1(K,w,A)$, where $\pi_1(K,w,A)$ satisfies the conditions in Theorem \ref{thm:exactly2formula}.

Then 
\begin{equation*}
\gamma_n(G)/\gamma_{n+1}(G)\cong\bar{\mathscr{G}}_{1n}\times\bar{\mathscr{G}}_{2n}\times\dots\times\bar{\mathscr{G}}_{qn}\times\bar{\mathscr{G}}_{\infty n},
\end{equation*}
where the $\bar{\mathscr{G}}_{jn}$ are as defined in \cite[p.~174]{gaglione1975factor}.

In particular, each $\bar{\mathscr{G}}_{jn}$ is a finite abelian $p_j$-group for some prime $p_j$. Also, each $\bar{\mathscr{G}}_{\infty n}$ is a free abelian group of rank $R_n^\infty$.

The formula for $R_n^\infty$ is as follows. Write $G$ in the form $G(1)*\dots*G(s)$, where each $G(i)$ is of the form shown in Equation \eqref{eq:cgen}. Let $m_0=0$. Suppose that $m_i$ of the generators $c_1,c_2,\dots,c_{n_i}$ ($i=1,2,\dots,s$) have infinite order (see \eqref{eq:g1gs}, \eqref{eq:cgen}). 

Let
\begin{equation*}
z=\sum_{k=1}^\infty x^k=\frac{1}{1-x}-1.
\end{equation*}

Let
\begin{align*}
U_\infty(x)&=1+(1+z)^{m_s}\left\{(s-1)-\sum_{j=1}^s(1+z)^{-(m_j-m_{j-1})}\right\},\\
\alpha_{\infty n}&=-\frac{1}{n!}\left\{\frac{d^n}{dx^n}\log[1-U_\infty(x)]|_{x=0}\right\}.
\end{align*}

Then
\begin{equation*}
R_n^\infty=\begin{cases}
m_s &\text{for $n=1$,}\\
\frac{1}{n}\displaystyle\sum_{\substack{k\mid n\\ k>1}}[\mu(\frac{n}{k})][k\alpha_{\infty k}] &\text{for $n>1$},
\end{cases}
\end{equation*}
where $\mu(i)$ is the M\"{o}bius function \cite[p.~179]{hall2018theory}.
\end{theorem}
\begin{proof}
When $(K,w,A)$ satisfies the conditions of Theorem 2.8, $G=\pi_1(K,w,A)$ is a free product of finitely generated abelian groups. Hence Theorem 2.1 and Theorem 2.2 in \cite[p.~175]{gaglione1975factor} hold. When $n=1$, $\gamma_n(G)/\gamma_{n+1}(G)=\Ab(G)$. With the help of Corollary \ref{cor:abfreeprods} and Equation \eqref{eq:cgen}, we see that the rank $R_n^\infty$ is equal to $m_s$, the total number of generators with infinite order.
\end{proof}

\begin{prop}
\label{prop:r1graph}
Let $(K,w,A)$ be a weighted graph. Let $G=\pi_1(K,w,A)$ be written in the form of \eqref{eq:g1gs}, \eqref{eq:cgen}. Let $m_i$ be the number of generators among $c_1, c_2,\dots,c_{n_i}$ that have infinite order.

Then
\begin{equation*}
R_1^\infty=m_s=\rank(\pi_1(K))+\left|\{ab\in A\mid w(ab)=0\}\right|.
\end{equation*}
We also note that $R_1^\infty$ is the rank of (the torsion-free part of) $\Ab(G)$.
\end{prop}
\begin{proof}
By Corollary \ref{cor:pi1weightedgraph} and the fact that $\pi_1(K)$ is a free group since $K$ is a graph, we have that $m_s$ is the sum of the rank of $\pi_1(K)$ and the number of edges $ab$ in the maximal tree $A$ satisfying $w(ab)=0$. The result then follows from Theorem \ref{thm:gaglione}.

We also note that $R_1^\infty$ is the rank of $\bar{\mathscr{G}}_{\infty 1}$, which is the torsion-free part of $\gamma_1(G)/\gamma_2(G)=\Ab(G)$.
\end{proof}

The following lemma is a useful result regarding $\gamma_2(G)=[G,G]$.

\begin{lemma}
\label{lemma:normalclosure}
Let $G$ be a finitely generated group with generators $c_1,\dots,c_n$. Let $N$ be the normal closure of the set $S=\{[c_i,c_j]\mid 1\leq i<j\leq n\}$.

Then, $\gamma_2(G)=N$.
\end{lemma}
\begin{proof}
Since $\gamma_2(G)$ is a normal subgroup containing $S$, hence $N\leq\gamma_2(G)$. On the other hand, in $G/N$ the generators commute hence $G/N$ is abelian. Thus $N$ must contain the commutator subgroup, that is, $\gamma_2(G)=[G,G]\leq N$.
\end{proof}

We show an example by computing the lower central series by hand and showing that it agrees with the theorems presented previously.

\begin{eg}
Consider the weighted graph $(K,w,A)$ in Figure \ref{fig:circle}, Example \ref{eg:circleK}. Let $w_{01}=0$ and $w_{12}=2$. Then

\begin{equation*}
\begin{split}
G&=\pi_1(K,w,A)\\
&\cong\Z*\Z*\Z/2\\
&\cong\langle c_1\rangle*\langle c_2\rangle*\langle c_3\mid c_3^2=1\rangle.
\end{split}
\end{equation*}

By Proposition \ref{prop:abgraph} or Corollary \ref{cor:abfreeprods}, we have $\Ab(G)\cong\Z\oplus\Z\oplus\Z/2$.

We see that 

\begin{equation*}
R_1^\infty=m_3=2
\end{equation*}

is equal to 

\begin{equation*}
\rank(\pi_1(K))+\left|\{ab\in A\mid w(ab)=0\}\right|=1+1=2
\end{equation*}

as predicted by Proposition \ref{prop:r1graph}.

For $n=2$, we can use Theorem \ref{thm:gaglione} to calculate that

\begin{equation*}
\begin{split}
U_\infty(x)&=1+\frac{2x-1}{(1-x)^2}\\
a_{\infty 2}&=1\\
R_2^\infty&=\frac{1}{2}[\mu(\frac{2}{2})][2a_{\infty 2}]=1.
\end{split}
\end{equation*}

We can also compute $\gamma_2(G)/\gamma_3(G)$ directly. By Lemma \ref{lemma:normalclosure}, $\gamma_2(G)$ is the normal closure of $\{[c_1,c_2],[c_1,c_3],[c_2,c_3]\}$. We note that in the quotient group $\gamma_2(G)/\gamma_3(G)$, all conjugates of a fixed commutator $[c_i,c_j]$ are in the same coset. For instance, for any $g\in G$, we have $g^{-1}[c_1,c_2]g\gamma_3(G)=[c_1,c_2]\gamma_3(G)$ since $[c_1,c_2]^{-1}g^{-1}[c_1,c_2]g=[[c_1,c_2],g]\in\gamma_3(G)$.

Therefore, 
\begin{equation*}
\gamma_2(G)/\gamma_3(G)=\langle [c_1,c_2]\gamma_3(G), [c_1,c_3]\gamma_3(G), [c_2,c_3]\gamma_3(G)\rangle.
\end{equation*}

The first generator $[c_1,c_2]\gamma_3(G)$ has infinite order. The other two generators have order 2 due to the relation $c_3^2=1$. For instance, $[c_1,c_3]^2=[[c_3,c_1],c_3]\in\gamma_3(G)$.

Therefore, 
\begin{equation*}
\gamma_2(G)/\gamma_3(G)\cong\Z\oplus\Z/2\oplus\Z/2.
\end{equation*}
We see that $R_2^\infty=1$, which agrees with the calculation using Theorem \ref{thm:gaglione}.
\end{eg}

\section{Tracking Location of ``Birth'' and ``Death'' of Cycles}
\label{sec:track}
In persistent homology \cite{Zomorodian2005,edelsbrunner2012persistent}, studying the ``birth'' and ``death'' of a homology class is of great importance. In practice, the point cloud data could be divided into several regions, and each cycle could lie in any of the regions. For instance, in the study of the brain using persistent homology \cite{bendich2016persistent,Lee2012}, it could be possible that the points are divided into regions based on different locations of the brain (e.g.\ left brain, right brain, etc.). However, just from looking at the persistent homology groups or the barcodes it is not possible to tell the location of the cycle which is born or has died. 

By using suitable weights, the weighted fundamental group itself can contain some information that enables us to tell the location of the cycles which are born or has died. We illustrate this using the following example.

\begin{eg}
Consider the three WSCs in Figure \ref{fig:3filtration}, which form a filtration. The points in Figure \ref{fig:3filtration} are divided into two regions -- left and right. We may let the weights of edges on the left region be 2 and the weights of edges on the right region be 3. To be precise, we let 
\begin{align*}
w(v_0v_1)&=w(v_0v_2)=w(v_1v_2)=2,\\
w(v_2v_3)&=w(v_2v_4)=w(v_3v_4)=3.
\end{align*}

We can calculate that
\begin{align*}
\pi_1(K_0,w,A_0)&=\Z*\Z/2*\Z/2,\\
\pi_1(K_1,w,A_1)&=\Z*\Z*\Z/2*\Z/2*\Z/3*\Z/3,\\
\pi_1(K_2,w,A_2)&=\Z*\Z/2*\Z/2*\Z/2*\Z/3*\Z/3.
\end{align*}

The presence of the free factor $\Z/3*\Z/3$ and an extra copy of the free factor $\Z$ in $\pi_1(K_1,w,A_1)$ tells us that in $K_1$, a cycle is \emph{born} in the right region. Meanwhile, the disappearance of a free factor $\Z$ and the appearance of an extra copy of the free factor $\Z/2$ in $\pi_1(K_2,w,A_2)$ allows us to conclude that a cycle has \emph{died} in the left region.

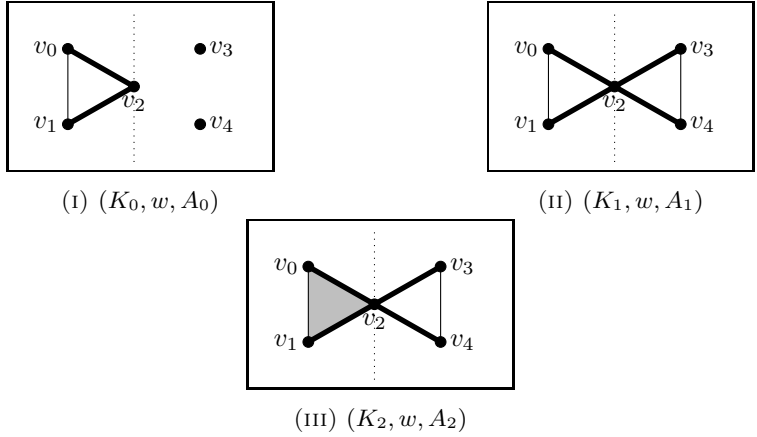
\begin{figure}[htbp]
\begin{subfigure}{0.49\textwidth}
\centering
\fbox{
\begin{minipage}{0.49\textwidth}
\begin{tikzpicture}[scale=1]
\filldraw 
(0,0) circle (2pt) node[below] {$v_2$}
(0.87,0.5) circle (2pt) node[right] {$v_3$}
(0.87,-0.5) circle (2pt) node[right] {$v_4$}
(-0.87,0.5) circle (2pt) node[left] {$v_0$}
(-0.87,-0.5) circle (2pt) node[left] {$v_1$};
\draw[line width=2pt] (0,0)--(-0.87,0.5);
\draw[line width=2pt] (0,0)--(-0.87,-0.5);
\draw (-0.87,0.5)--(-0.87,-0.5);
\draw[dotted] (0,-1)--(0,1);
\end{tikzpicture}
\end{minipage}
}
\caption{$(K_0,w,A_0)$}
\end{subfigure}
\begin{subfigure}{0.49\textwidth}
\centering
\fbox{
\begin{minipage}{0.49\textwidth}
\begin{tikzpicture}[scale=1]
\filldraw 
(0,0) circle (2pt) node[below] {$v_2$}
(0.87,0.5) circle (2pt) node[right] {$v_3$}
(0.87,-0.5) circle (2pt) node[right] {$v_4$}
(-0.87,0.5) circle (2pt) node[left] {$v_0$}
(-0.87,-0.5) circle (2pt) node[left] {$v_1$};
\draw[line width=2pt] (0,0)--(-0.87,0.5);
\draw[line width=2pt] (0,0)--(-0.87,-0.5);
\draw (-0.87,0.5)--(-0.87,-0.5);
\draw[line width=2pt] (0,0)--(0.87,0.5);
\draw[line width=2pt] (0,0)--(0.87,-0.5);
\draw (0.87,0.5)--(0.87,-0.5);
\draw[dotted] (0,-1)--(0,1);
\end{tikzpicture}
\end{minipage}
}
\caption{$(K_1,w,A_1)$}
\end{subfigure}
\begin{subfigure}{0.49\textwidth}
\centering
\fbox{
\begin{minipage}{0.49\textwidth}
\begin{tikzpicture}[scale=1]
\fill[fill=lightgray]
(0,0)  
-- (-0.87,0.5)
-- (-0.87,-0.5);
\filldraw 
(0,0) circle (2pt) node[below] {$v_2$}
(0.87,0.5) circle (2pt) node[right] {$v_3$}
(0.87,-0.5) circle (2pt) node[right] {$v_4$}
(-0.87,0.5) circle (2pt) node[left] {$v_0$}
(-0.87,-0.5) circle (2pt) node[left] {$v_1$};
\draw[line width=2pt] (0,0)--(-0.87,0.5);
\draw[line width=2pt] (0,0)--(-0.87,-0.5);
\draw (-0.87,0.5)--(-0.87,-0.5);
\draw[line width=2pt] (0,0)--(0.87,0.5);
\draw[line width=2pt] (0,0)--(0.87,-0.5);
\draw (0.87,0.5)--(0.87,-0.5);
\draw[dotted] (0,-1)--(0,1);
\end{tikzpicture}
\end{minipage}
}
\caption{$(K_2,w,A_2)$}
\end{subfigure}
\caption{The three WSCs form a filtration $K_0\subseteq K_1\subseteq K_2$. Their respective maximal trees $A_0$, $A_1$, $A_2$ are marked in bold.}
\label{fig:3filtration}
\end{figure}
\end{eg}

We remark that weighted persistent homology defined in \cite{ren2018weighted} can also tell when a cycle (containing a special point) is formed or has disappeared. In view of this, the main advantage of using weighted fundamental groups is its potential connection to deeper aspects of algebraic topology. For instance, it is possible for weighted fundamental groups to be non-abelian (Examples \ref{eg:circleK} and \ref{eg:circleL}) while weighted (persistent) homology groups in \cite{ren2018weighted} are always abelian.

Another solution to the problem of tracking the location of cycles is to compute not only the ranks but also the generators of the persistent homology barcodes, as done in the seminal paper \cite{busaryev2010tracking} by O.\ Busaryev, T.\ Dey and Y.\ Wang. In their paper, the goal is to track a chosen essential generating cycle via reordering simplices in the filtration. A motivating application includes scanning objects where the shape is represented by a discrete sample, and inferring geometrical and topological properties from such data. The paper also includes an effective algorithm that is tested on three point cloud models with simplicial complexes of variable sizes.

\section{Application to the study of chemical molecules}
\label{sec:chem}
We outline some applications of the weighted fundamental group to the study of molecules.
\subsection{Weights to distinguish between types of bonds.}
\label{sec:typesofbonds}
In chemistry there are different types of bonds between atoms. For example, single or double bonds between carbon atoms are common in molecules.

Fullerenes are an important class of molecules. Fullerenes are also considered as nanomaterials, which is a significant area of research in materials science \cite{valiev2002materials,gottschalk2009modeled,martin1994nanomaterials}. Remarkably, a number of chemical properties of a fullerene can be derived from its graph structure \cite{schwerdtfeger2015topology}.

A classical example of a fullerene is $\text{C}_{60}$. The fullerene $\text{C}_{60}$ has two bond lengths \cite[p.~364]{katz2006fullerene}. The bonds between two hexagons can be considered ``double bonds'' and are shorter than the bonds between a hexagon and a pentagon. 

The weighted fundamental group can be used to study molecules by assigning suitable weights to different types of bonds. For instance, single and double bonds can be assigned different weights. We illustrate the idea by the following example.

\begin{eg}
Consider the two WSCs in Figure \ref{fig:tworings}, which represents the pentagon and hexagon rings in $\text{C}_{60}$. By setting appropriate weights, the weighted fundamental group is able to distinguish between the two ring structures. For instance, we may let $w_K(\sigma)=1$ for all edges $\sigma\in K$. Then $\pi_1(K,w_K,A)=\Z$.

On the other hand, we may let $w_L(v_5v_6)=w(v_9v_{10})=w(v_7v_8)=2$ (these edges represent the double bonds), while $w_L(\tau)=1$ for all other edges $\tau\in L$. Then $\pi_1(L,w_L,B)=\Z*\Z/2*\Z/2*\Z/2$.
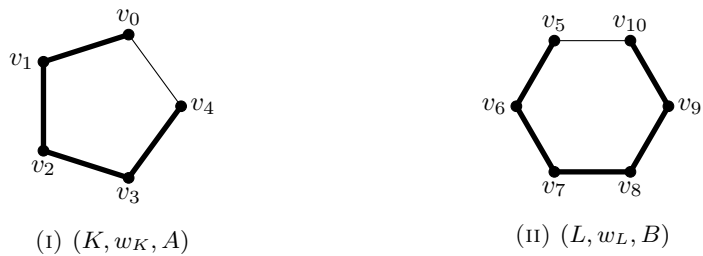
\begin{figure}[htbp]
\begin{subfigure}{0.49\textwidth}
\centering
\begin{tikzpicture}[scale=1]
\filldraw 
(1,0) circle (2pt) node[right] {$v_4$}
(0.31,0.95) circle (2pt) node[above] {$v_0$}
(-0.81,0.59) circle (2pt) node[left] {$v_1$}
(-0.81,-0.59) circle (2pt) node[below] {$v_2$}
(0.31,-0.95) circle (2pt) node[below] {$v_3$};
\draw (1,0)--(0.31,0.95);
\draw[line width=2pt] (0.31,0.95)--(-0.81,0.59);
\draw[line width=2pt] (-0.81,0.59)--(-0.81,-0.59);
\draw[line width=2pt] (-0.81,-0.59)--(0.31,-0.95);
\draw[line width=2pt] (0.31,-0.95)--(1,0);
\end{tikzpicture}
\caption{$(K,w_{K},A)$}
\end{subfigure}
\begin{subfigure}{0.49\textwidth}
\centering
\begin{tikzpicture}[scale=1]
\filldraw 
(1,0) circle (2pt) node[right] {$v_9$} 
(0.5,0.87) circle (2pt) node[above] {$v_{10}$}
(-0.5,0.87) circle (2pt) node[above] {$v_5$}
(-1,0) circle (2pt) node[left] {$v_6$}
(-0.5,-0.87) circle (2pt) node[below] {$v_7$}
(0.5,-0.87) circle (2pt) node[below] {$v_8$};
\draw[line width=2pt] (1,0)--(0.5,0.87);
\draw (0.5,0.87)--(-0.5,0.87);
\draw[line width=2pt] (-0.5,0.87)--(-1,0);
\draw[line width=2pt] (-1,0)--(-0.5,-0.87);
\draw[line width=2pt] (-0.5,-0.87)--(0.5,-0.87);
\draw[line width=2pt] (0.5,-0.87)--(1,0);
\end{tikzpicture}
\caption{$(L,w_{L},B)$}
\end{subfigure}
\caption{The 2 WSCs representing the pentagon and hexagon rings of $\text{C}_{60}$ respectively. Their respective maximal trees $A$, $B$ are marked in bold.}
\label{fig:tworings}
\end{figure}
\end{eg}

\subsection{Distinguishing between Hamiltonian Paths}
\label{sec:hamiltonian}
In graph theory, a Hamiltonian path is a path that visits each vertex of the graph exactly once. Hamiltonian paths have some important applications in biology. In 2018, Twarock, Leonov and Stockley \cite{twarock2018hamiltonian} used Hamiltonian path analysis (HPA) to study viral genomes. Hamiltonian paths on fullerenes have also been studied \cite{maruvsic2007hamilton}.

A Hamiltonian path can be viewed as a maximal tree by considering the union of all the edges and vertices in the Hamiltonian path: It is path-connected and does not contain cycles (otherwise a vertex would be visited more than once) and therefore a tree. Since it contains all vertices, it is a maximal tree.

The weighted fundamental group is an algebraic invariant that can detect different maximal trees when the edges are suitably weighted (see Remark \ref{remark:choicemaxtree}). For instance, we may give each edge a different integer weight. Hence, the weighted fundamental group can be used to distinguish between different Hamiltonian paths.

\subsection{Alternative approaches}
The problems in Sections \ref{sec:typesofbonds} and \ref{sec:hamiltonian} have potential to be analyzed by computing the generators of persistent homology barcodes, as done in \cite{busaryev2010tracking}. For distinguishing between types of bonds, firstly information regarding each bond type can be stored in a separate database. Then, tracking a chosen generating cycle in \cite{busaryev2010tracking} can correspond to detecting a chosen pentagon or hexagon ring of a molecule. For Hamiltonian paths, tracking a generator (containing particular edges) in the style of \cite{busaryev2010tracking} can correspond to distinguishing between different Hamiltonian paths that contain or do not contain those particular edges.

\subsection{Algorithmic complexity}
The maximal tree $A$ in a simplicial complex $K$ can be given by Kruskal's algorithm which has a time complexity of $O(E\log E)$, where $E$ is the number of edges in $K$ \cite{sorensen2005algorithm}. The presentation of the weighted fundamental group is then given by listing the generators and the respective relations given in Definition \ref{defn:relations}. We remark that due to the Novikov-Boone theorem \cite{baumslag1959some}, there exists a finite presentation of a group for which there is no algorithm that, given two words $a$, $b$, determines whether $a$ and $b$ are the same element in the group. In view of the above (negative solution to the word problem for groups), we remark that it can be difficult to analyze or simplify presentations of the weighted fundamental group. 

For comparison, the time complexity for persistent homology is cubic time (in the size of the complex) \cite{wagner2012efficient,boissonnat2014computing}.

\section*{Acknowledgements}
We wish to thank the referees most warmly for numerous suggestions that have improved the exposition of this paper.

\bibliographystyle{amsplain}
\bibliography{jabref9}

\end{document}